\documentclass[a4paper,11pt]{article}
\usepackage[utf8]{inputenc}
\usepackage[T1]{fontenc}

\usepackage{bbm}
\usepackage{amsthm}
\usepackage{tikz}
\usepackage{mathrsfs,amssymb} 
\usepackage{enumitem}
\usepackage{fullpage}
\usetikzlibrary{arrows}
\usetikzlibrary{arrows.meta}
\usepackage{wrapfig}
\usepackage{mathtools}
\usepackage{amssymb}
\usepackage[nomath]{lmodern}
\usepackage{graphicx}
\usepackage{pgf,tikz,tkz-graph,subcaption}
\usetikzlibrary{arrows,shapes}
\usetikzlibrary{decorations.pathreplacing}
\usepackage{tkz-berge}
\usepackage{enumitem}
\usepackage[normalem]{ulem}
\usepackage{hyperref}
\usepackage{comment}
\hypersetup{colorlinks = true, linkcolor = blue, citecolor = blue, urlcolor = blue}
\usepackage[capitalise]{cleveref}

\allowdisplaybreaks

\usepackage[margin=1in]{geometry}
\usepackage{bookmark}

\newtheorem{theorem}{Theorem}
\newtheorem*{theorem*}{Theorem}

\newtheorem{cor}[theorem]{Corollary}
\newtheorem{lemma}[theorem]{Lemma}

\newtheorem{obs}[theorem]{Observation}
\newtheorem{prop}[theorem]{Proposition}

\newtheorem{claim}[theorem]{Claim}
\newcommand*{\myproofname}{Proof}
\newenvironment{claimproof}[1][\myproofname]{\begin{proof}[#1]}{\end{proof}}

\DeclareMathOperator{\sat}{sat}
\DeclareMathOperator{\wsat}{ssat}
\DeclareMathOperator{\ex}{ex}

\newcommand{\F}{\mathcal{F}}
\newcommand{\p}{\mathcal{P}}
\newcommand{\h}{\mathcal{H}}

\newcommand*{\floorfrac}[2]{\mathopen{}\left\lfloor\frac{#1}{#2}\right\rfloor\mathclose{}}
\newcommand*{\abs}[1]{\lvert #1\rvert}
\newcommand{\floor}[1]{\left\lfloor{#1}\right\rfloor}
\newcommand{\ceil}[1]{\left\lceil{#1}\right\rceil}

\title{Hypergraph saturation for the bow tie} 

\author{Stijn Cambie\thanks{Department of Computer Science, KU Leuven Campus Kulak-Kortrijk, 8500 Kortrijk, Belgium. Supported by the Research Foundation Flanders (FWO) (grant number 1225224N). E-mail: {\tt stijn.cambie@hotmail.com}} \and Nika Salia\thanks{King Fahd University of Petroleum and Minerals, Dhahran, Saudi Arabia. 
supported by the National Research, Development and Innovation Office NKFIH, grants K132696, E-mail: {\tt salianika@gmail.com}}}

\begin{document}
\parindent=0cm

\maketitle

\begin{abstract}
Erdős and Sós initiated the study of the maximum size of a $k$-uniform set system, for $k \geq 4$, with no singleton intersections $50$ years ago. In this work, we investigate the dual problem: finding the minimum size of a $k$-uniform hypergraph with no singleton intersections, such that adding any missing hyperedge forces a singleton intersection. These problems, known as saturation and semi-saturation, are typically challenging. Our focus is on an elementary-to-state case in the line of work by Erdős, Füredi and Tuza. We establish tight linear bounds for $k=4$, marking one of the first non-obvious cases with such a bound.
\end{abstract}

\section{Introduction}\label{sec:intro}

A (hyper)graph is said to be $\p$-saturated if it lacks a specific property $\p$ but acquires this property upon the addition of any new edge. 
Many problems in graph theory focus on determining the maximum or minimum number of edges that a graph with $n$ vertices can have while remaining saturated for a particular property $\p$.
One of the most studied examples is the Turán number, denoted as $\ex(n,F)$. 
This parameter represents the maximum number of edges in a graph $G$ of order $n$ that is $F$-free, which equals the maximum size of an $F$-saturated graph (here $\p$ is the absence of any subgraph isomorphic to $F$).
Saturation problems have wide applications in combinatorics and can also be relevant in other fields, such as bootstrap percolation~\cite{BBMR12}.

The study of minimum saturated (hyper)graphs was first introduced by Erdős, Hajnal and Moon in 1964 \cite{erdos1964problem}. 
The saturation number of a graph $H$ is the minimum number of edges in an $H$-saturated graph of order $n$, and for a $k$-uniform hypergraph $H$ it is defined as follows:

\[
\sat_k(n, H) = \min\{e(G) : G \text{ is an $H$-saturated $k$-uniform hypergraph of order } n\}.
\]

It is also interesting to explore this question without requiring $G$ to be $H$-free. In that case, we say that $G$ is $H$-semi-saturated. 
The minimum number size of an $n$-vertex $k$-uniform $H$-semi-saturated hypergraph is known as the \emph{semi-saturation number}, denoted by $\wsat_k(n,H)$.
Trivially, $\wsat_k(n,H)\le \sat_k(n,H)$, since a saturated (hyper)graph is also semi-saturated. In this paper, we do not consider other variants, such as weak saturation, induced saturation and saturation for Berge hypergraphs.
We refer interested readers to~\cite{faudree2011survey,AC19,bulavka2023weak,diskin2024saturation,kalai1985hyperconnectivity,kronenberg2021weak,moshkovitz2015exact,pikhurko2001weakly,shapira2023weakly}.

Bollobás determined the exact saturation number for complete $k$-uniform hypergraphs~\cite{bollobas1986extremal}. 
Erdős, Furedi and Tuza~\cite{erdHos1991saturated} and Pikhurko~\cite{pikhurko2000asymptotic,pikhurko2004results}
obtained results for the saturation number of complete $k$-uniform stars $S^k_{1,m-1}$.
Pikhurko~\cite{pikhurko1999minimum} proved that for any finite family of forbidden $k$-uniform hypergraphs $\F$, $\sat_k(n,\F)=O(n^{k-1})$, this implies a conjecture of Tuza~\cite{tuza1986generalization, tuza1988extremal}.
It is unknown if the condition $\F$ being finite is necessary to obtain this result.

In 2023, English, Kostochka and Zirlin~\cite{EKZ23} studied the saturation of a three-uniform linear cycle of length three. 
Specifically, let $C_3^3$ be a three-uniform hypergraph with six vertices and three hyperedges, where each pair of hyperedges shares exactly one common vertex, and no vertex belongs to all three hyperedges. They obtained the following bounds:
\[
\left(\frac{4}{3} + o(1) \right)n \leq \sat_3(n, C_3^3) \leq \frac{3}{2}n + O(1).
\]
The lack of knowledge regarding the saturation number of cycles highlights the hardness of saturation problems. 

Erdős, F{\"u}redi and Tuza~\cite{EFT91} determined the saturation number for families $H_k(k + 1, k)$, $H_k(2k - 2, 2)$ and $H_k(k+1, 3)$, where $H_k(p, q)$ represents the family of all $k$-uniform hypergraphs with $p$ vertices and $q$ edges. 
By determining the saturation and semi-saturation numbers of the bow tie $B_k=H_k(2k-1,2)$, a $k$-uniform hypergraph with $2k-1$ vertices and two hyperedges sharing one vertex, we consider a missing elementary but hard case (for $k \ge 4$).

For $k = 2$, $\floorfrac{n}{2}$ independent edges form a $B_2$-(semi)-saturated graph and since there are no two isolated vertices in any $B_2$-(semi)-saturated graph, we conclude that
\[
\wsat_2(n, B_2) =\sat_2(n, B_2)= \floorfrac{n}{2}.
\]
For $k = 3$, $n$-vertex hypergraphs that are $B_3$-(semi)-saturated can be obtained by taking the union of disjoint hyperedges $K_3^{(3)}$, and one complete 3-uniform hypergraph with $4$ vertices $K_3^{(3)}$ if $3\mid n-1$ or a $S_{2,3}^{3}$ ($3$ hyperedges sharing a fixed pair of vertices) if $3\mid n-2$.

By the elementary observations on the components of the hypergraph (e.g. a component of order $n'$ has size at least $\ceil{\frac{n'-1}{2}}$), it is easy to deduce that those constructions are minimum examples and
\[
\wsat_3(n, B_3)=\sat_3(n,B_3) = \floorfrac{n}{3} +3\cdot \mathbbm{1}_{3|n-1}+2\cdot \mathbbm{1}_{3|n-2}.
\]
In this paper, we present two exact results for $4$-uniform hypergraphs. Specifically, we determine the exact values of the saturation and semi-saturation numbers for $4$-uniform bow ties. See~\cref{sec:main} for the overview of our results.
Note that determining the minimum size of a $B_k$-saturated hypergraph, which is a maximal $k$-uniform hypergraph without $B_k$,
answers a dual of a problem by Erdős and Sós~\cite{erdHos1975problems}.
They conjectured that the maximum number of hyperedges in a $k$-uniform hypergraph without a bow-tie $B_k$ is $\binom{n}{k-2}$, for $k\geq 4$ and $n$ large enough. This was solved by Frankl~\cite{frankl1977families} and strengthened for $k=4$ by Keevash, Mubayi and Wilson~\cite{KMW06}.

\subsection{Definitions and notation}\label{sec:def&not}

Let $\h$ be a hypergraph, we denote the vertex set of $\h$ by $V(\h)$ and the hyperedge set of $\h$ by $E(\h)$.
Let $v$ be a vertex of $\h$, then $\deg(v)$ denotes the degree of $v$, the number of hyperedges of $\h$ incident with $v$.
The neighborhood of $v$, the set of vertices adjacent to $v$, is denoted by $N(v)$; the closed neighborhood of vertex $v$ is denoted by $N[v]$ and is equal to $N(v)\cup \{v\}$.
A pair of vertices incident to the same set of hyperedges are called \emph{twin vertices}.
The hypergraph is twin-free if it has no twin vertices.

For integers $n\leq m$, we use $[n..m]$ to denote the set $\{n,n+1,\ldots,m\}$ and write $[n]=[1..n].$

The generalized Johnson graph $J:= J(n, k, 1)$ is a graph with vertex set $\binom{[n]}{k}$, where two vertices are adjacent if and only if the intersection of the corresponding $k$-sets contains exactly one element.
Note that a semi-saturated hypergraph corresponds to a dominating set of $J$, a set such that every edge of $J$ is incident to at least one vertex of the dominating set.
A saturated hypergraph corresponds to an independent dominating set of $J$, a set such that every edge of $J$ is incident to exactly one vertex of the set.

\section{Main results}\label{sec:main}
As we have seen in the introduction for $k \in \{2,3\},$ the values $\wsat_k(n, B_k)$ and $\sat_k(n, B_k)$ are easily determined, and are equal to $\frac{n}{k}+ O(1).$
Since there are no more than $k-1$ isolated vertices, $\wsat_k(n, B_k)$ and $\sat_k(n, B_k)$ are at least linear in $n$ for every $k.$
By considering the union of $K^{(k)}_{2k-2}$s and $K^{(k)}_{2k-3}$s, we know that the saturation number is linear for every $n$.
Here we show the asymptotic solution for semi-saturation for all $k$.
\begin{prop}
 For $k\geq 4$, we have 
 \[
 \frac{2n}{k}+O(1) \ge \wsat_k(n,B_k) \ge \frac{2n}{k+O(1)}.
 \]
\end{prop}

\begin{proof}
 First, we prove an upper bound.
 Consider a $k$-regular graph $G$ of girth at least $k+1$ (for existence, see e.g.~\cite{Sauer70}).
 Let $H_G$ be the dual of $G$.
 That is, $H_G$ is the hypergraph whose vertices are the edges of $G$, i.e., $V(H_G)=E(G)$, and whose hyperedges are the $k$-tuples of edges of $G$ sharing a fixed vertex.
 Let $H$ be an $n$-vertex $k$-uniform hypergraph obtained from $\floor{\frac{n}{\abs{V(H_G)}}}$ disjoint copies of $H_G$ and a $k$-uniform complete hypergraph with $n-\abs{V(H_G)}\floor{\frac{n}{\abs{V(H_G)}}}$ vertices.
 For each $h\in\binom{V(H)}{k}$, consider the $k$-edge subgraph of $G$ formed by the edges that correspond to the vertices of the hyperedge $h$.
 Since the girth of $G$ is large, this $k$-edge subgraph of $G$ is a forest, which means it contains a leaf vertex $v$ and a leaf edge $e$. 
 Therefore, the hyperedge $h$ intersects the hyperedge in $H$ corresponding to $v$ at exactly one vertex. Thus $H$ is $B_k$-semi-saturated, implying the desired upper bound $\abs{E(H)} \le \frac{2n}{k}+O(1).$

 In the other direction, as $k \geq 4$, fewer than $k$ hyperedges may have at least 3 vertices of degree one. Otherwise, one could add a hyperedge that contains exactly two or three degree-one vertices from such hyperedges, without creating a new copy of $B_k$. 
 This implies that 
\[
k\lvert E(H) \rvert= \sum_{v \in V(H)} \deg(v) \geq 2n - k^2 -2\lvert E(H) \rvert \Rightarrow \lvert E(H) \rvert \geq \frac{2n}{k + 2}-O(1). 
\]
\qedhere
\end{proof}

For $k=4$, we determine the exact semi-saturation number of $B_k$ for infinitely many $n$. The proof of the following theorem is detailed in the succeeding section,~\cref{sec:ssat}. The lower bound is established using an averaging argument involving a degree function, while the upper bound is achieved through a construction based on modifying the dual hypergraph of an auxiliary regular graph with a large girth. 
 
\begin{theorem}\label{thm:semi_sat}
For $n\ge 100$, $\wsat_4(n,B_4)\geq \frac{6}{13}n-\frac{11}{13}$ and this is sharp for $n \equiv 4 \pmod{13}$
\end{theorem}

We also determine the saturation number of $B_4$ for every $n$. One can compare these values for the minimum size of $B_4$-saturated hypergraphs with the corresponding maximum size obtained in~\cite[Thm.~1.1]{KMW06}.

\begin{theorem}\label{thm:b_4_sat}
 \[
  \sat_4(n,B_4)= \begin{cases}
 \binom{n}{4} &\mbox{ if } n\leq 6,\\
 n &\mbox{ if } n \in \{7,8\},\\
 n-3 &\mbox{ if } n\geq 9.
 \end{cases}
 \]
\end{theorem}
Note that if $n<7$, Theorem~\ref{thm:b_4_sat} trivially holds.
The case $n=7$ is elementary and elegant. The $n=8$ case was checked by a computer program. These proofs are presented next.
The lower bound for Theorem~\ref{thm:b_4_sat} for $n \ge 9$ will be deduced from~\cref{Lem:core} (proven in~\cref{sec:b_4_sat}) in~\cref{cor:sat}. Both are stated at the end of this section.

\begin{proof}[Proof of~\cref{thm:b_4_sat} for $n=7$ and $n=8$]
 We first prove that $\sat_4(7, B_4) \geq 7$. Assume, for contradiction, that $\sat_4(7, B_4) < 7$. This implies there are at most $6$ hyperedges. Since each hyperedge is part of 4 bow ties in $K_7^{(4)}$ there must exist a set of four vertices that do not form a hyperedge and is not part of a bow tie with any of the 6 hyperedges, as $\binom{7}{4} - 6 - 4 \cdot 6 > 0$, a contradiction.

 Next, we show that there exists a unique $4$-uniform $B_4$-saturated $7$-vertex hypergraph.
 Following the previous argument, since $7+7\cdot 4=\binom{7}{4},$ no two hyperedges participate in a bow tie with the same missing hyperedge. 
 For two distinct hyperedges $A$ and $B$ we have
 $\abs{A \cap B}\neq 3,$ equivalently
 $\abs{A^c \cap B^c}\neq 2,$ since they both form a bow tie with $(A \cap B)^c$;
 $\abs{A \cap B}\neq 1,$ equivalently $\abs{A^c \cap B^c}\neq 0,$ since they form a bow tie.
 Thus, for every pair of distinct hyperedges $A$ and $B$, we have $\abs{A^c \cap B^c} = 1$. Consequently, the set consisting of the complements of the hyperedges forms a Steiner system $S(2,3,7)$, the Fano Plane $FP$. 
 We denote the unique $B_4$-saturated hypergraph of order and size $7$ by $FP^c.$


 Here we show $\sat_4(8,B_4)\leq 8$, the hypergraph $H$ satisfying 
 \[
 V(H)=[8]\mbox{ and } E(H)=\left\{A \in \binom{[6]}{4} \colon \abs{A \cap \{1,2,3,4\}} \in \{2,4\} \right\} \cup \{\{5,6,7,8\}\},
 \]
 is an $8$-vertex $4$-uniform $B_4$-saturated hypergraph of size $8$.
 Hypergraph $H$ has three vertex orbits; $[4],\{5,6\}$ and $\{7,8\}$ and thus $\frac{8!}{4! \cdot 2! \cdot 2!}=420$ automorphisms.
 Since there are $420$ minimum size maximal independent sets (of size $8$) in $J(8,4,1)$, as computed in \cite[doc. \text{sat(8,4)}]{C24github}, we conclude that $\sat_4(8,B_4)= 8$ and $H$ is the unique $B_4$-saturated hypergraph of order and size~$8$. \qedhere
\end{proof}

\begin{figure}[h]
 \centering
 \begin{tikzpicture}[scale=0.99]
 \foreach \x in {45,135,...,315}{\draw[fill] (\x:1.5) circle (0.15);
}
\draw[dashed] (0,0) circle (1.75);

 \foreach \x in {-2,2}{
 \foreach \y in {-1.5,1.5}{
 \draw[fill] (\x,\y) circle (0.15);
}
}

 \foreach \x in {-2.5,2.5}{
 \draw[fill] (\x,0) circle (0.15);
}

\draw[dotted] (4,0.7)--(4,-0.7);

\draw[dashed] (0.7,1.75) -- (2.25,1.75)--(2.25,-1.75)--(0.7,-1.75) -- cycle;

\draw[dashed] (-0.7,1.75) -- (-2.25,1.75)--(-2.25,-1.75)--(-0.7,-1.75) -- cycle;

\draw[dashed, color=blue] (0.8,1.5) -- (2.75,0.5)--(2.75,-2.25)--(0.8,-1.25) -- cycle;

\draw[dashed, color=blue] (0.8,-1.5) -- (2.75,-0.5)--(2.75,2.25)--(0.8,1.25) -- cycle;

\draw[dashed, color=blue] (-0.8,1.5) -- (-2.75,0.5)--(-2.75,-2.25)--(-0.8,-1.25) -- cycle;

\draw[dashed, color=blue] (-0.8,-1.5) -- (-2.75,-0.5)--(-2.75,2.25)--(-0.8,1.25) -- cycle;

\foreach \y in {-1.5,1.5}{
 \draw[fill] (4,\y) circle (0.15);
}

\draw[dashed, color=red] (1.75,2.25) -- (1.75,-2)--(4.5,-1.65) -- cycle;

\draw[dashed, color=red] (1.75,-2.25) -- (1.75,2)--(4.5,1.65) -- cycle;

 \end{tikzpicture} 
 
 \caption{Saturated $B_4$-free hypergraph of order $n \ge 12$ and size $n-3$}\label{fig:Sat_n_ge12}
 \end{figure}
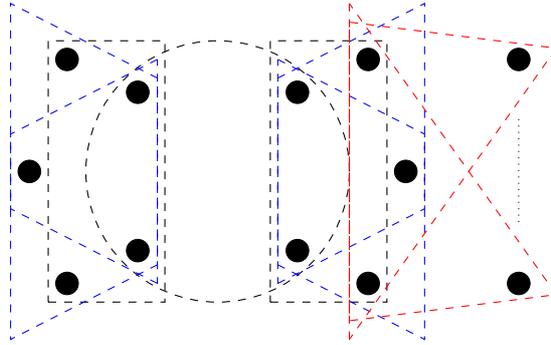

\begin{proof}[Proof of the upper bound~\cref{thm:b_4_sat} for $n \ge 9$]
For $n\in \{9,11\}$, we take $K_{4}^{(4)} \cup K_{5}^{(4)}$ and $K_{4}^{(4)} \cup FP^c. $
 For $n=10$, define $H$ by
 $V(H)=[10]$ and $$E(H)=\left\{A \in \binom{[5]}{4} \colon \abs{A \cap \{1,2,3\}}=2 \right\} \cup \{4,5,6,7\} \cup \left \{A \in \binom{[6..10]}{4} \colon \abs{A \cap \{8,9,10\}}=2 \right\}.$$
 For $n \geq 12$, partition the vertex set $[n]$ into three sets $[10], V,$ and $W$, where $\abs{V}, \abs{W} \neq 1$. Consider the hypergraph consisting of the hyperedges $\{1,2,3,v\}$ and $\{8,9,10,w\}$, where $v \in V$ and $w \in W$, along with the hyperedges of $H$, which is the construction on 10 vertices. This forms an $n$-vertex, size $n-3$, $B_4$-saturated hypergraph. The construction is illustrated for $n=12$ and $V = \emptyset$ in~\cref{fig:Sat_n_ge12}.

This concludes the proof of the upper bound for $\sat_4(n, B_4)$ for all $n$.
\end{proof}

The following lemma is a key tool for proving the matching lower bound for $\sat_4(n, B_4)$, which we prove in Section~\ref{sec:b_4_sat}.

\begin{lemma}\label{Lem:core}
 Let $H$ be a connected $B_4$-free, $B_4$-saturated $4$-uniform hypergraph of order $n$. Then $\abs{E(H)} \geq n-3$. When equality holds, there exists a $B \in \binom{[n]}{3}$ such that no $A \in E(H)$ satisfies $\abs{A \cap B} = 1$. If $H$ is twin-free and $n \geq 5$, the stronger inequality $\abs{E(H)} \geq n$ holds.
\end{lemma}

The key argument in the proof of this lemma involves studying a maximum tree $T$ such that the graph $H_j$ defined by $V(H_j)=\{v_i, v'_i \mid v_i \in V(T)\}$ and 
$E(H_j)=\{ (v_i, v'_i, v_p, v'_p) \mid v_iv_p \in E(T) \}$ is a subhypergraph (after relabeling) of $H.$ Here $H_j$ is a kind of two-fold blowup of $T$.
Also, we will use the core $C(T) \subset T$ containing all $v_i \in V(T)$ for which $\{v_i,v'_i\}$ is a pair of twin vertices in $H$.

As a corollary of~\cref{Lem:core}, we conclude with the proof for the lower bound of~\cref{thm:b_4_sat}, which finishes the proof of the whole theorem.

\begin{cor}\label{cor:sat}
 For every $n \ge 4$, $\sat_4(n,B_4) \ge n-3$.
\end{cor}

\begin{proof}
 Let $H$ be a $B_4$-free $B_4$-saturated $4$-uniform hypergraph with $n$ vertices. Then so are each of its components.
 This implies that at most one of its components can contain twin vertices; otherwise, a hyperedge could be added that includes twin vertices from two components without creating $B_4$, a contradiction.
 Thus, by Lemma~\ref{Lem:core}, all components except for one (distinct from isolated vertex component) have at least as many edges as vertices.
 Thus if there are no isolated vertices we conclude.
 
 If there are $2$ or $3$ isolated vertices, no component contains twin vertices and we conclude. 
 If there is one isolated vertex $v$ and a component $C$ with $\abs{E(C)}=\abs{V(C)}-3,$ the triple $B \subset V(C)$ from Lemma~\ref{Lem:core} extended with $v$ can be added as an edge to $H$ without creating a bow tie, a contradiction.
Hence in all cases, we have $\abs{E(H)} \ge n-3$.
\end{proof}

Note that Corollary~\ref{cor:sat}, together with the constructions provided above, concludes the proof of Theorem~\ref{thm:b_4_sat}. The proof of Lemma~\ref{Lem:core} is presented in Section~\ref{sec:b_4_sat}.





\section{Semi-saturation of the 4-uniform bow tie}\label{sec:ssat}

 \begin{proof}[Proof of upper bound for Theorem~\ref{thm:semi_sat}]

\begin{claim}\label{clm:SharpCon}
 For every $r \in \{4,5,6\}$, there exists a $B_4$-semi-saturated connected twin-free $4$-uniform hypergraph with $n=13r$ vertices and $6r=\frac{6}{13}n$ hyperedges.
 \end{claim}
 \begin{claimproof}
 \begin{figure}[h]
\centering
\begin{tikzpicture}
 \foreach \x in {0,1,2,3}{
 \draw[thick] (0,0)--(90*\x+45:1.5);
 \draw[fill=black] (90*\x+45:1.5) circle (3pt);
 }
 \draw[fill=black] (0,0) circle (3pt);
 \node (u) at (0,0.3) {$u_i$} ;
 \end{tikzpicture}\quad
 \begin{tikzpicture}
 \foreach \x in {0,1}{
 \draw[thick] (-0.75,0)--(60*\x+150:2.12);
 \draw[fill=black] (60*\x+150:2.12) circle (3pt);
 }
 \foreach \x in {0,1}{
 \draw[thick] (0.75,0)--(60*\x-30:2.12);
 \draw[fill=black] (60*\x-30:2.12) circle (3pt);
 } 
 \draw[thick] (-0.75,0)--(0.75,0);

\draw[fill=black] (-0.75,0) circle (3pt);
\draw[fill=black] (0.75,0) circle (3pt);
 \node (u) at (-0.7,0.3) {$v_i$} ;
 \node (u) at (0.7,0.3) {$w_i$} ;
 
 \end{tikzpicture}
 
 \caption{Local modification used in~\cref{clm:SharpCon}}
 \label{fig:connecting2graphs}
\end{figure}
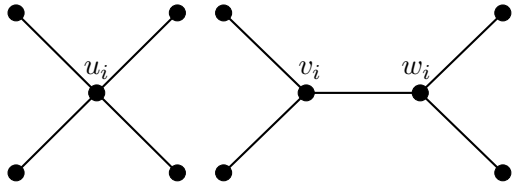
 Take the graph from~\url{https://houseofgraphs.org/graphs/50403}~\cite{HOG}, or the ones with HoG id $1138$ or $33766$. 
 These are $4$-regular graphs of order $5r$ (with $r \in \{4,5,6\}$), girth $5$ and whose square graph has independence number~$r$ (verified in \cite[{doc. 4-regular girth 5 construction}]{C24github}). 
 That is, its vertex set can be partitioned in $r$ disjoint closed neighborhoods $(N[u_i])_{i \in [r]}$.
 Now delete each vertex $u_i$, $i \in [r]$ and replace them with two vertices $v_i, w_i$, that are adjacent to each other and two distinct pairs of vertices from $N(v_i)$, see Figure~\ref{fig:connecting2graphs}. 
 The resulting graph $G$ has $6r$ vertices and $11r$ edges. 
 Now construct the $4$-uniform hypergraph $H$ for which 
 \[
 V(H)=E(G)\cup \{v_i\}_{i \in [r]} \cup \{w_i\}_{i \in [r]} \] 
 and the hyperedge set is defined in the following way.
 \begin{itemize}
 \item For every vertex $x$ of degree $4$ in $G$, the set containing four edges incident with $x$ in $G$ is a hyperedge of $H$.
 \item For every vertex $x \in \{v_i, w_i\}, i \in [r]$, of degree $3$ in $G$, the set containing three edges incident with $x$ in $G$, together with $x$ is a hyperedge of $H$. 
 \end{itemize}
 The $4$-uniform hypergraph $H$ has $13r$ vertices, $11r$ vertices of degree $2$ (corresponding to the edges of $G$), and $2r$ degree one vertices (corresponding to the degree three vertices of $G$), and $6r$ hyperedges.
 The hypergraph $H$ is connected and twin-free, since no two vertices of degree one are adjacent and every degree two vertex is a center of an induced bow tie.
 As the girth of $G$ is at least $5$, the hypergraph $H$ is $B_4$-saturated. 
 \end{claimproof}

 Finally, for every $n\geq 108$ with $n\equiv 4 \pmod{13}$, there is a hypergraph $H_n$, consisting of disjoint copies of the hypergraphs from \cref{clm:SharpCon} and an isolated hyperedge.
 The hypergraph $H_n$ is a $B_4$-semi-saturated hypergraph 
 with $n$ vertices and $\frac{6}{13}(n-4)+1 =\frac{6n-11}{13}$ hyperedges.\end{proof}

\begin{proof}[Proof of lower bound for Theorem~\ref{thm:semi_sat}]
Let $H$ be a saturated $4$-uniform hypergraph.
We refer to hyperedges containing a pair of twin vertices of degree one or two, as \textit{special}.
The other hyperedges are referred to as \textit{normal}.
Since adding the union of two pairs of twin vertices does not create a bow tie, each such union has to belong to $H$, and hence $H$ has no more than $3$ special hyperedges.
 
 For every hyperedge $e\in E(H)$ we define 
$ \phi(e)=\sum_{v \in e} \frac{1}{\deg(v)}.$
 We call a hyperedge for which $\phi(e)>\frac{13}6$ \textit{heavy}.

 \begin{obs}
 For any heavy normal hyperedge $e\in E(H)$, the degree sequence of its vertices is $(1,2,2,x)$ for some $x\in[2..5].$
 \end{obs}

 \begin{claim}\label{clm:neighbouring_edges_withlow_phi}
 Let $e = \{a, b, c, d\}$ be a heavy normal hyperedge of $H$, where $\deg_H(a) = 1$ and $\deg(b) = \deg(c) = 2$. Let $f$ and $h$ be distinct hyperedges in $E(H)$ different from $e$, that contain $b$ and $c$, respectively. Then, $f$ and $h$ are normal hyperedges, and 
 $\min\{\phi(f), \phi(h)\} \leq 2.$
 \end{claim}

 \begin{claimproof}
 If both hyperedges $f$ and $h$ contain degree $1$ vertices $v_f$ and $v_h$, then $\{v_f,v_h,b,c\}$ is not a hyperedge of $H$ and it intersects three hyperedges of $H$ each in two vertices, a contradiction. Thus we have $\min\{\phi(f), \phi(h)\} \leq 2.$ 
 
 If $f$ is a special hyperedge, it contains twin vertices $u$ and $v$, which are different from $b$ since $e$ is not a special hyperedge. Then, $\{a, b, u, v\}$ is not a hyperedge of $H$ and does not intersect any hyperedge in exactly one vertex, leading to a contradiction with $H$ being $B_4$-saturated. Thus $f$ is normal hyperedge, and by the same reasoning, $h$ is a normal hyperedge too.
 \end{claimproof}

 If the degree sequence of a normal hyperedge $h$ is $(1,2,2,2),$ by~\cref{clm:neighbouring_edges_withlow_phi} $h$ has two incident normal hyperedges $h_1,h_2$ with $\phi(h_1),\phi(h_2)\leq 2$. 
 We associate hyperedges $h_1$ and $h_2$ with $h$.
 If the degree sequence of a normal hyperedge $h$ is $(1,2,2,x)$ for some $x\in \{3,4,5\},$ by~\cref{clm:neighbouring_edges_withlow_phi} it has an incident hyperedge $h_1$ with $\phi(h_1)\leq 2$.
 We associate hyperedge $h_1$ with $h$.
 
 \begin{obs}
 No hyperedge is associated with two heavy normal hyperedges.
 \end{obs}
 \begin{claimproof}
 For the sake of contradiction, assume that two heavy normal hyperedges, $e$ and $f$, are associated with the same hyperedge $h$. Then, the set consisting of the degree 1 vertices from $e$ and $f$, together with the degree 2 vertices from $e \cap h$ and $f \cap h$, is not a hyperedge of $H$. Adding this set to $H$ does not create a bow tie, which leads to a contradiction.
 \end{claimproof} 
 
 As we have seen, every heavy normal hyperedge is associated with a normal hyperedge or a pair of normal hyperedges. 
 Such that the average of the function $\phi$ over them is upper bounded by 
 \begin{equation}\label{eq:average_weight}
 \frac{(1+\frac{1}{2}+\frac{1}{2}+\frac{1}{2})+2+2}{3},\frac{(1+\frac{1}{2}+\frac{1}{2}+\frac{1}{3})+2}{2}=\frac{13}{6}.
 \end{equation}

 With a simple case analysis, we also note that the sum of $\phi$ over the special hyperedges is bounded by $\frac 92$ in general and trivially by $4$ if there is only one special hyperedge. 
 Let $n'$ be the number of vertices belonging to at least one hyperedge.
 By~\eqref{eq:average_weight} and the previous remark, we have
 \begin{equation}\label{eq:bound_for_n'}
 n'=\sum_{e \in E(H)} \phi(e) \le \frac{13}{6}\abs{E(H)}+\frac{11}{6}.
 \end{equation}

 Note that there are at most $3$ isolated vertices. 
 If $H$ contains an isolated vertex, then no incident hyperedge of a heavy normal hyperedge contains a degree $1$ vertex. 
 Therefore, the upper bound in Equation~\eqref{eq:average_weight} will be significantly improved, to $\frac{3\cdot 2+2.5}4=\frac{17}{8}$, leading to the corresponding improvement in the linear term of Equation~\eqref{eq:bound_for_n'}; $n'\le \frac{17}{8}\abs{E(H)}$.
 For $n\ge 63$, we conclude that
 \[
 \abs{E(G)} \ge \frac{8}{17}(\abs{V(G)}-3) \ge \frac{6}{13}\abs{V(G)} -\frac{11}{13} 
 \]
 \qedhere
\end{proof}

\section{Saturation of for the 4-uniform bow tie}\label{sec:b_4_sat}

In this section, we prove Lemma~\ref{Lem:core}.

\begin{proof}[Proof of Lemma~\ref{Lem:core} ]
We construct a spanning subhypergraph of $H$ using a greedy approach:

\textbf{Step 1:} Begin by selecting a single edge to form the initial hypergraph $H_1$.

\textbf{Step 2:} While the current hypergraph $H_i$ is not yet spanning, update it by adding a hyperedge $e$ from $H$ to obtain $H_{i+1}$, with the following conditions:
\begin{itemize}
 \item $H_{i+1}$ remains connected.
 \item The chosen hyperedge $e$ maximizes the order of the resulting hypergraph $H_{i+1}$.
\end{itemize}

\textbf{Step 3:} Repeat this Step 2, until the hypergraph spans all vertices.

Since $H$ is connected, the procedure will terminate with a spanning subhypergraph of $H$.
Since $H$ is connected and $B_4$-free, each additional hyperedge from Step 2 contains one or two new vertices, thus we have $\lvert V(H_{i+1}) \rvert - \lvert V(H_i) \rvert \leq 2$, for all $i$.

\begin{claim}\label{Claim:Structure_of_sceleton_tree}
 Let $H'$ and $H''$ be two vertex-disjoint subhypergraphs of $H$, each constructed following the same procedure described in Steps 1-2, with one exception: a hyperedge is added only if it intersects the existing hypergraph in exactly two vertices.

 Then there exists a hypergraph $H'''$ such that
 \begin{itemize}
 \item $H'''\subseteq H$;
 \item $H',H''\subseteq H'''$;
 \item $H'''$ can be obtained in the same way as $H'$ and $H''$.
 \end{itemize}
\end{claim}
\begin{claimproof}
 Since $H$ is connected, there exists a shortest path $P$ from $H'$ to $H''$. As $H$ is $B_4$-free, it follows that $P$ consists of a sequence of hyperedges, where each consecutive pair of hyperedges shares exactly two vertices, any other pair of hyperedges are vertex disjoint and the two vertices from the terminal hyperedges belong to $H'$ and $H''$, respectively.
 The hypergraph obtained from the union of the hyperedges $P$, $H'$ and $H''$ has all three required characteristics. 
\end{claimproof}

By Claim~\ref{Claim:Structure_of_sceleton_tree}, let $j$ be minimal index such that $\lvert H_{j+1} \rvert - \lvert H_j \rvert = 1$, if such a $j$ exists. Then, for every $j' > j$, we have $\lvert H_{j'+1} \rvert - \lvert H_{j'} \rvert = 1$. 
 
Let $j$ be the maximal index for which $\lvert H_j \rvert = \lvert H_{j-1} \rvert + 2$; if no such index exists, then set $j = 1$. 
In this latter case where $j=1$, every hyperedge contains at least $3$ out of the $4$ vertices of $H_1$, which without loss of generality is the hyperedge $[4]$.
Since $H$ is $B_4$-saturated we have $E(H)=\left\{ A \in \binom{[n]}{4} \colon \abs{A \cap [4]} \ge 3 \right\},$ which has size $4(n-4)+1 \ge n$ as $n\ge 5$. (For all $n<5$ $\abs{E(H)} \ge n-3$.)
From now on, we assume $j > 1$. 
During the construction of $H_j$ from $H_1$, we added two vertices at each step. Since $H$ is $B_4$-free, the two vertices added at each step are twin vertices in $H_j$. Even more, by Claim~\ref{Claim:Structure_of_sceleton_tree}, all twin vertices of $H$, belong to~$H_j$.

We now construct an auxiliary tree $T$, whose vertices correspond to the pairs of twin vertices in $H_j$, and an edge is placed between two pairs if their union forms a hyperedge of $H_j$. 
Note that every edge in $H$ intersects $T$ in at least three (of the original) vertices of $H.$
Also $\abs{E(T)}=j \geq 2$.
Let the core $C(T)$ be the set of twin pairs of vertices of $H$, for which the two vertices in the pair correspond to a vertex of $T$.

 \begin{claim}\label{clm:badleaves&nb_leaves}
 For a pair $(v_i,u_i) \in V(T)\setminus C(T)$ we have
 \begin{itemize}
 \item the vertex $(v_i,u_i)$ of $T$ is a leaf of $T$, or
 \item the vertex $(v_i,u_i)$ of $T$ has degree two in $T$, and has a neighbour which is a leaf in $T$, or
 \item $T=S_4$ and $(v_i,u_i)$ is a center of the star.
 \end{itemize}
 \end{claim}
 \begin{claimproof}
 Let $e \in E(H)$ intersect $(v_i,u_i)$ in one of the two vertices.
 Let $(v_i',u_i')$ be a neighbour of $(v_i,u_i)$ in $T$. 
 Then since $H$ is $B_4$-free and $\{v_i,u_i,v_i',u_i'\}$ is a hyperedge of $H$, $e\cap (v_i',u_i')\neq \emptyset$. 
 Even more, if $(v_i',u_i')$ is not a leaf of $T$ then $e$ has at least two original vertices in the component of $T\setminus (v_i,u_i)$ with the vertex $(v_i',u_i')$. 
 Thus $e$ has exactly one vertex common with $(v_i,u_i)$, at least one vertex in common with every component of $T \setminus (v_i,u_i)$, and at least two vertices in common for each component of $T \setminus (v_i,u_i)$ with at least two vertices. 
 Thus either $(v_i,u_i)$ is a leaf or has degree two and one neighbor which is a leaf, or $T=S_4$ and $(v_i,u_i)$ is a center of it. 
 \end{claimproof}
 
 As a corollary of~\cref{clm:badleaves&nb_leaves}, we note that if $C(T) \not= \emptyset$, $T \setminus C(T)$ is a union of isolated vertices and edges.
 We now prove that there are at least as many hyperedges containing at least one vertex outside $C(T)$, as there are vertices in $H$ outside $C(T)$, by considering three cases.

 \begin{itemize}
 \item Let $p_1p_2$ be an edge of $T \setminus C(T)$.
 Let $p_1=(c,d)$ be a leaf, $p_2=(a,b)$ be its neighbour and $p_3=(u,v)$ be a distance $2$ neighbour belonging to $C(T)$.
 The edge that intersects $p_2$ in one vertex must contain both $u,v$ as $H$ is $B_4$-free and $u,v$ are twins. 
 It also intersects $p_1$.
 Thus without loss of generality, we may assume $\{a,c,u,v\} \in E(H).$
 
 If no edge contains exactly three vertices from $\{a,b,c,d\},$ then every hyperedge intersecting $\{a,b,c,d\}$ contains $u,v$. Even more, since $H$ is $B_4$-saturated, we have \[
 \left\{f \cup \{u,v\}:f\in \binom{\{a,b,c,d\}}{2} \right\} \subset E(H).
 \]
 In this case, $B=\{a,u,v\}$ is a triple intersecting no hyperedge in a single vertex.
 
 If there is an edge containing exactly three vertices from $\{a,b,c,d\},$ every such edge is of the form $\{a,b,c,f\},$ for some vertex $f$ since $H$ is $B_4$-free. This implies immediately that no hyperedge intersects $B=\{a,b,c\}$ in a single vertex.
 As $H$ is $B_4$-saturated, we also know that $\{b,c,u,v\} \in E(H)$.

 In both cases, we have at least four hyperedges of $H$ of the form $\binom{ \{a,b,c,d,u,v\} }4$ and containing a vertex $a$ or $b$.

\item 
Let $p_1=(a,b)$ an isolated vertex of $T \setminus C(T)$. 
Let $p_1=(a,b)$ be a leaf of $T$ and and $p_2=(u,v)$ be its neighbor from $C(T)$.
There is a hyperedge that intersects $p_1$ in one vertex, without loss of generality we can assume it is $\{a,u,v,c\}$ for some vertex $c$.
Now every other hyperedge either contains $\{u,v\}$ or $\{a,b,c\}.$
This implies that also $\{u,v,b,c\} \in E(H)$ as $H$ is $B_4$-saturated, and that no hyperedge from $E(H)$ intersects $B=\{a,u,v\}$ in a single vertex.
Note that $c$ is not a part of a $K_2$ of $T \backslash C(T).$

If $c \not\in T,$ (which has to be read as $c \not \in H_j$) mark it as a special vertex.
All such special vertices and all vertices belonging to a component $K_1$ of $T \setminus C(T)$ are incident to at least two hyperedges containing exactly two of these vertices and two of the core vertices. 

\item Every vertex $c \not\in T$ which is not special, belongs to at least one hyperedge containing $c$ and three vertices of $H_j$.
\end{itemize}

Now we are ready to lower bound the number of hyperedges. 
Note that our goal is to show that the number of hyperedges is at least $n-3$. 
We know that for every vertex not from the core $C(T)$, we have at least one hyperedge on average. 
Observe that if there are two distinct pairs of twin vertices from $C(T)$, then their union is a hyperedge of $H$, as $H$ is $B_4$-saturated. 
Thus we have at least $\binom{\abs{C(T)}}{2}$ hyperedges containing vertices only from $C(T)$. 
As $\binom{\abs{C(T)}}{2}-2\abs{C(T)}\geq -3$, the number of hyperedges spanned by the vertices of $C(T)$, is at most three less than the number of vertices of $H$ in $C(T)$. Thus we are done if $H$ is twin-free

 From here we assume there is no pair of twin vertices in $H$, thus we have $C(T)= \emptyset.$
 By~\cref{clm:badleaves&nb_leaves}, we deduce that $T \in \{ P_3,P_4,S_4\},$ as $\abs{E(T)} \geq 2$.

 The elementary computer verification for $P_3, P_4$ and $S_4$ finishing the following explanation each time can be found in~\cite{C24github}. 
 The construction of $T$ was arbitrary, $H$ might have different such trees, from here we assume that $T$ was constructed in such a way that it maximized the diameter of~$T$.

 \begin{itemize}

 \item \textbf{case $T=P_3$.} 
 As $T$ was contracted with the maximum possible diameter, there are no two disjoint hyperedges, by Claim~\ref{Claim:Structure_of_sceleton_tree}.
 We may assume that $H$ contains the edges $[4]$ and $[3..6]$.
 The other hyperedges belong to $\binom{[6]}{4}$, or intersect $[6]$ in $3$ elements.
 For every $v \in [7..n]$ one can consider the hyperedges containing $v$ after removing $v$, $E_v=\{e \setminus v \colon e \in E(H), v \in e\} \subset \binom{[6]}{3}.$
 Let $S= \cup_{v \in [7..n]} E_v.$
 Given $S$, one can count the number $y$ of hyperedges in $\binom{[6]}{4}$ that intersect each edge of $S$ in at least $2$ vertices (which have to belong to the saturated hypergraph).
 Now $\abs{E(H)}-n= \sum_{v \in [7..n]} \left(\abs{E_v}-1\right) + y-6.$
 Let $x=\sum_{v \in [7..n]} \left(\abs{E_v}-1\right).$
 Recall that $S$ has no disjoint edges.
 
 If $S$ contains some $3$-edges that intersect in one vertex, we know that they belong to the same $E_v$.
 If there are $z$ pairs of edges in $S$ intersecting in one vertex, we know that $\binom{x}{2} \ge z$.
 Hence, for every choice of $S$, we can compute $z$ and $y$ and derive a lower bound for $x+y,$ which turns out to be at least $6$ in all cases, proving that $m \ge n.$ The latter has been done with an elementary computer program (\cite[doc. \text{P\_3\_case}]{C24github}).
 \item \textbf{case $T=P_4$.}
 Let $H_3$ be the hypergraph with vertex set $[8]$ and edges $[4],[3..6]$ and $[5..8].$
 Since $C(T)=\emptyset,$ there are hyperedges in $H$ intersecting $\{3,4\}$ and $\{5,6\}$ in exactly one element.
 By symmetry, we can either assume that $\{1,3,5,7\} \in E(H)$, or $\{1,3,5,6\}$ and $\{3,4,5,7\}$ are both in $E(H).$
 The hyperedges containing $3$ elements of $[8]$ intersect none of the present hyperedges in exactly one vertex. By considering $H_3$, the intersection with $[8]$ is thus of the form $\{1,3,4\},[2..4],[5..7]$ or $\{5,6,8\}.$
 By symmetry and the additional present edges, $[2..4]$ and $\{5,6,8\}$ cannot appear and three cases remain for the set $Y$ of triples which are the intersection of a hyperedge in $E(H)$ and $[8]$.
 Taking the minimum size of a maximal independent set in a subgraph of the Johnson graph $J(8,4,1)$ in all situations, we conclude that $m \ge n$ in all cases, more precisely there are at least $8$ hyperedges in $H$ among the ones in $\binom{[8]}4.$
 \item \textbf{case $T=S_4$.} 
 We may assume that $E(H_3)=\{ \{1,2,3,4\},\{1,2,5,6\},\{1,2,7,8\}\}.$
 Since $\{1,2\}$ is not a pair of twins, we can assume without loss of generality that $\{1,3,5,7\} \in E(H).$
 Every hyperedge which is not spanned by vertices $[8]$, has three neighbors in $[8]$, thus it contains either vertex $1$ or vertex $2$, thus it should intersect all three hyperedges of $T$ in at least three vertices, thus every such hyperedge contains vertices $[2]$. 
 Given that neighborhood, one can compute the number of hyperedges present in the saturated hypergraph and conclude exactly as in the previous case where $T=P_4.$ \qedhere 
 \end{itemize}
\end{proof}

\section*{Acknowledgement}

We thank S. English for sharing his expert knowledge on hypergraph saturation problems, and D. Chakraborti and A. Sgueglia for their early interest in the topic.


\bibliographystyle{abbrv}
\bibliography{ref}

\end{document}